\documentclass[12pt, reqno]{amsart}

\pdfoutput=0

\usepackage{caption}

\usepackage{setspace}

\usepackage{verbatim}

\usepackage{enumerate, enumitem}
\setenumerate[1]{label=(\arabic*)}

\usepackage{amsfonts,fancyhdr,ifthen,bm}
\usepackage{amssymb,amsthm,amsmath}
\usepackage[letterpaper,left=1in,right=1in,top=1in,bottom=1in]{geometry}
\usepackage{times}
\usepackage{graphicx}
\usepackage{color}

\newtheorem{thm}{Theorem}[section]
\newtheorem*{thm*}{Theorem}
\newtheorem{prop}[thm]{Proposition}

\newtheorem*{cor*}{Corollary}
\newtheorem{lem}[thm]{Lemma}

\newtheorem*{quest*}{Question}
\newtheorem*{arg*}{Argument Principle}
\newtheorem*{tar*}{Tarski-Seidenberg Theorem}
\newtheorem*{cad*}{Cylindrical Algebraic Decomposition}
\newtheorem*{dt*}{Dimension Theorem}

\theoremstyle{remark}

\theoremstyle{remark}
\newtheorem*{rmk*}{Remark}

\theoremstyle{definition}
\newtheorem{defn}[thm]{Definition}

\theoremstyle{definition}
\newtheorem*{defn*}{Definition}

\numberwithin{equation}{section}

\newcommand{\Imag}{\textrm{Im}}
\newcommand{\rat}{\mathcal{R}}

\newcommand{\poly}{\mathcal{P}}
\newcommand{\masses}{\mathcal{M}}
\newcommand{\zar}[1]{\text{Zariski}(#1)}
\renewcommand{\dim}[1]{\text{dim}(#1)}

\begin{document}

\title[Counting Zeros of Harmonic Rational Functions]{Counting Zeros of Harmonic Rational Functions and Its Application to Gravitational Lensing}

\author[P. M.  ~Bleher]{Pavel M. Bleher$^1$}
\author[Y. ~Homma]{Youkow Homma$^{1,2}$}
\author[L. L. ~Ji]{Lyndon L. Ji$^{1,2}$}
\author[R. K. W. ~Roeder]{Roland K.\ W.\ Roeder$^1$}

\markboth{\textsc{P. Bleher, Y. Homma, L. Ji, and R. Roeder}}
  {\textit{Counting Zeros of Harmonic Rational Functions and Its Application to Gravitational Lensing}}

\footnotetext[1]{
Department of Mathematical Sciences, Indiana University-Purdue University Indianapolis, 402 N. Blackford St., Indianapolis, IN 4
6202, USA,
and}
\footnotetext[2]{Carmel High School, 520 E. Main St., Carmel, IN 46032, USA.}

\begin{abstract}
General Relativity gives that finitely many point masses between an observer and a light source create many images of the light source. Positions of these images are solutions of $r(z)=\bar{z},$ where $r(z)$ is a rational function. We study the number of solutions to $p(z) = \bar{z}$ and $r(z) = \bar{z},$ where $p(z)$ and $r(z)$ are polynomials and rational functions, respectively. Upper and lower bounds were previously obtained by Khavinson-\'{S}wi\c{a}tek, Khavinson-Neumann, and Petters. Between these bounds, we show that any number of simple zeros allowed by the Argument Principle occurs and nothing else occurs, off of a proper real algebraic set. If $r(z) = \bar{z}$ describes an $n$-point gravitational lens, we determine the possible numbers of generic images.
\end{abstract}
\maketitle

\section{Introduction}

One of the results of Einstein's General Theory of Relativity is that a point mass placed between an observer and a light source will create two images of the source. If this single mass is replaced with a distribution of masses, significantly more complicated configurations of images can be created. Multiple images were first observed by astronomers in the 1970's and further technological advancements pushed gravitational lensing as an important tool in astrophysics. Gravitational lensing has also become an exciting field of research in mathematical physics---see the recent beautiful surveys \cite{FASSNACHT, KHAVINSON-NEUMANN2, PETTERS_SURVEY, PETTERS-WERNER} and for deeper discussion, including the history, see \cite{PETTERS} and \cite{SCHNEIDER}.

Suppose that the distribution of mass is well-localized relative to the distances between the observer and the mass and relative to the distances between the masses and the light source. Images of the light source are described by solutions for $z$ to
\begin{equation}\label{CAUCHY INTEGRAL}
\bar{z} = \displaystyle\int_{\mathbb{C}} \frac{d\mu (\zeta)}{z-\zeta},
\end{equation}
where $\mu$ is a compactly supported measure describing the distribution of mass projected onto the plane through the center of mass perpendicular to the line from the observer to the light source. See, for example, \cite{FASSNACHT} and \cite{STRAUMANN}.

An early and important result in gravitational lensing, due to Burke \cite{BURKE}, is that if $\mu$ is a smooth mass distribution, then the number of solutions to \eqref{CAUCHY INTEGRAL} is odd. This was generalized by Petters \cite{PETTERS2} to the situation where $\mu$ has smooth density except at $g$ points. In this case, he showed that the number of solutions is congruent to $(g-1)$ modulo $2$. See also \cite[Thm. 1]{PETTERS-WERNER}.

In this paper, we will focus on the case of $n$ point masses. For $1\le j \le n$, let $\sigma_j$ be a positive mass located at $z_j$. In this case, $\mu = \displaystyle\sum_{j=1}^{n} \sigma_j\delta(z-z_j)$ and \eqref{CAUCHY INTEGRAL} simplifies to become 
\begin{equation}\label{LENS EQUATION}
\bar{z} = \sum_{j=1}^{n} \frac{\sigma_j}{z - z_j}.
\end{equation}
For the remainder of the paper, we refer to \eqref{LENS EQUATION} as the \textit{lens equation}.

It is also interesting to study variations of Equation~\eqref{LENS EQUATION}, replacing the sum on the right-hand side with a general polynomial $p(z)$ or rational functions $r(z)$. In the polynomial case, Khavinson and \'{S}wi\c{a}tek \cite{KHAVINSON-SWIATEK} used a clever application of the Fatou Lemma from holomorphic dynamics combined with the ``Argument Principle" \cite{ARGUMENT PRINCIPLE} to show that if $p(z)$ has degree $n$, then the number of solutions to $p(z) = \bar{z}$ is bounded above by $3n-2$. It was a delicate question of whether this upper bound was achieved for each $n$. Using Thurston's Theorem from rational dynamics, Geyer \cite{GEYER 1} proved the sharpness of this bound.

In the rational case, Khavinson and Neumann \cite{KHAVINSON-NEUMANN} used similar techniques as \cite{KHAVINSON-SWIATEK} to prove that if $r(z)$ has degree $n$, then the number of solutions is bounded above by $5n-5$. (Note that $r(z) = p(z)/q(z)$ has degree $n = \max(\deg p, \deg q)$.) Surprisingly, the sharpness of this bound for each $n$ had already been proved by Rhie \cite{RHIE}, using an explicit construction of an appropriate configuration of masses in the lens equation \eqref{LENS EQUATION}. 

Let $U$ be an open subset of $\mathbb{C}$. A function $f: U \to \mathbb{C}$ is called \emph{harmonic} if both the real and imaginary parts of $f$ are harmonic in the classical sense. A zero $z_0 = x_0 + iy_0$ of $f(x,y) = u(x,y) +  iv(x,y)$ is called \emph{simple} if the Jacobian $\mathcal{D}(z_0) = \det{\bigl|\begin{smallmatrix} u_x&u_y\\v_x&v_y\end{smallmatrix} \bigr|} \not= 0$. We call a polynomial $p(z)$ simple if all of the zeros of $p(z) - \bar{z}$ are simple. Similarly, we call a rational function $r(z)$ simple if all of the zeros of $r(z) - \bar{z}$ are simple. 

If $p(z)$ is a simple polynomial of $\deg p \ge 2$, then there is also an obvious lower bound on the number of solutions by $\deg p$, as a consequence of the ``Argument Principle." If $r(z) = p(z)/q(z)$ is a simple rational function of $\deg r \ge 2$, then there is a lower bound on the number of solutions depending on the degrees of $p$ and $q.$ For example, if $\deg p \le \deg q$, then the lower bound is $\deg r - 1$.

A special subcase of the rational case is obtained by considering rational functions of the form \eqref{LENS EQUATION} with all positive masses $\sigma_i$. Since Rhie's examples were constructed with positive masses, the upper bound of $5n-5$ is still achieved. Meanwhile, Petters \cite{PETTERS2} showed using Morse Theory that if $r(z)$ is a simple rational function of the form \eqref{LENS EQUATION}, then the lower bound on the number of images is $n+1$. 

We will look at each of the three cases mentioned above (polynomial, rational, and physical) from the perspective of ``parameter spaces" and with a motivation of understanding what numbers of solutions between the lower and upper bounds can occur generically. 

For the purposes of this paper, we will parameterize the space of polynomials by their coefficients, letting 
\begin{align*}
	\poly_n=\left\{a_nz^n+a_{n-1}z^{n-1}+\cdots+a_0\ |\ a_n\in\mathbb{C}\setminus\left\{0\right\}, a_j\in\mathbb{C}\ \text{for}\ 0\leq j\leq n-1\right\}\cong\mathbb{C}\setminus\left\{0\right\}\times\mathbb{C}^n.
\end{align*}

\begin{rmk*}
After performing a rotation and/or shift in $z$, one can suppose that $a_n\in\mathbb{R}_{+}$ and $a_{n-1}=0$ without affecting any of the statements below.  However, we consider the present definition of $\poly_n$ more natural.
\end{rmk*} 

\begin{thm}\label{POLYNOMIAL: POSITIVE MEASURES}
Let $S\poly_n (k)$ be the set of simple polynomials of degree $n\geq 2$ with $p(z)-\bar{z}$ having $k$ roots. Then $S\poly_n(k)$ is a non-empty, open subset of $\poly_n$ if and only if $k=n,n+2,\ldots,3n-2$. Furthermore, the set of non-simple polynomials, $N\poly_n$, is the complement of the union 
\begin{equation*}
\bigcup_{k=n, n+2, \ldots, 3n-2} S\poly_n(k)
\end{equation*} 
and is contained in a proper real algebraic (hence measure $0$) subset of $\poly_n$.
\end{thm}

\begin{figure}
  \includegraphics[height=5.5cm]{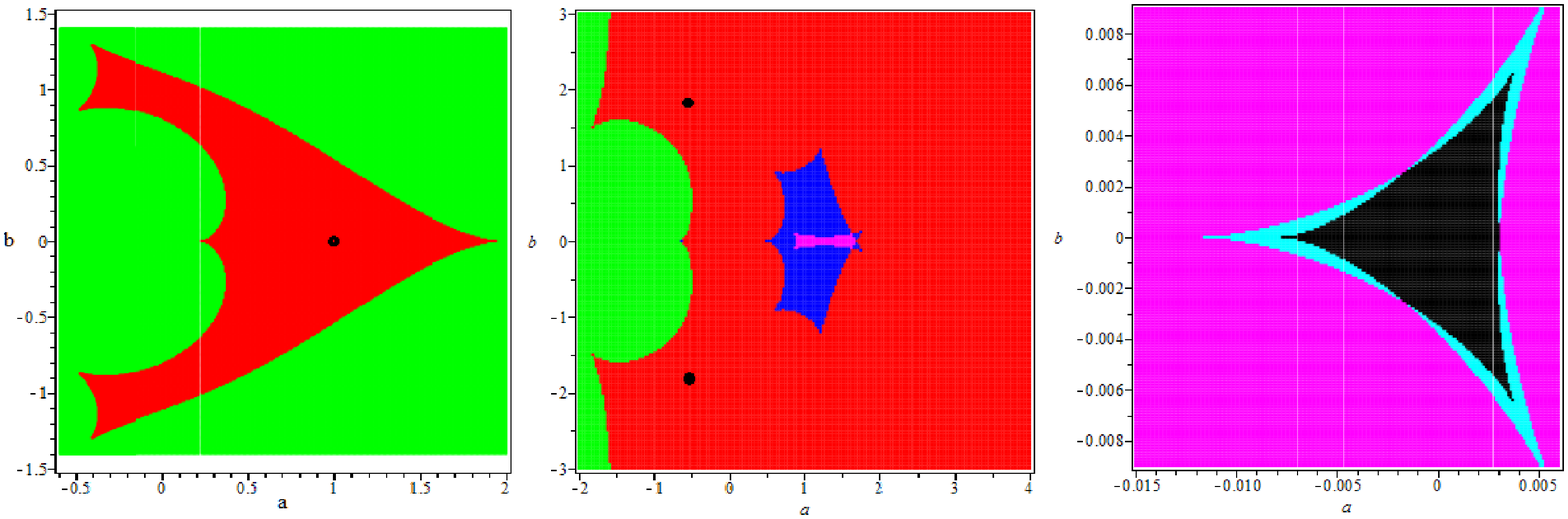}
	\caption[.]{Parameter space pictures for slices of $\masses_2$, $\masses_3$, and $\masses_4$, respectively. The fixed masses are denoted with large, black dots. 
	
	Left: There is one fixed mass located at $\left(1,0\right)$.  Placing a second, equal mass in the green region will produce 3 images and placing in the red region will produce 5 images. 
	
	Center: There are two equal, fixed masses located at $\left(-\frac{1}{2},\frac{\sqrt{3}}{2}\right)$ and $\left(-\frac{1}{2},-\frac{\sqrt{3}}{2}\right)$.  Placing a third, equal mass in the green region will produce 4 images, red- 6 images, blue- 8 images, and yellow- 10 images. 
	
	Right: There are three equal, fixed masses located at $\left(-\frac{1}{2},\frac{\sqrt{3}}{2}\right), \left(-\frac{1}{2},-\frac{\sqrt{3}}{2}\right),$ and $\left(1,0\right)$. Placing a fourth, equal mass in the magenta region will produce 11 images, light blue- 13 images, and black- 15 images.}
\label{Pictures4-2intro}
\end{figure}

We parameterize the space of rational functions of degree $n$ by their coefficients up to scalings, with the condition that $p$ and $q$ are relatively prime. More specifically, let
\begin{equation}\label{RATIONAL FUNCTION}
r(z) = \frac{a_n z^n + \ldots + a_0}{b_n z^n + \ldots + b_0}
\end{equation}
for $(a_n, \ldots, a_0, b_n, \ldots, b_0) \in \mathbb{C}^{2n+2}$.  The space $\rat_n$ of rational functions of degree $n$ can be parameterized by $2n+2$-tuples of complex numbers $(a_n, \ldots, a_0, b_n, \ldots, b_0)$, considered up to non-zero complex scaling with two restrictions: 

\begin{enumerate}
\item $a_n \neq 0$ or $b_n \neq 0$ and
\item The resultant of $a_n z^n + \ldots + a_0$ and $b_n z^n + \ldots + b_0$ is not equal to $0$.
\end{enumerate}
As such, it is an open subset of $\mathbb{CP}^{2n+1}$.

\begin{thm}\label{RATIONAL: POSITIVE MEASURES}
Let $S\rat_n(k)$ be the set of simple rational functions $r(z)$ of degree $n\geq 2$ such that $f(z) = r(z)-\bar{z}$ has $k$ roots. Then for all $n \ge 2$, $S\rat_n(k)$ is a non-empty, open subset of $\rat_n$ if and only if  $k = n-1, n+1, \ldots, 5n-5$. Moreover, the complement of 

\begin{equation}\label{RATIONAL COMPLEMENT}
\bigcup_{k=n-1, n+1, \ldots, 5n-5} S\rat_n(k),
\end{equation} 
is contained in a proper real algebraic (hence measure $0$) subset of $\rat_n$.
\end{thm}

\begin{rmk*}
In the rational case, the complement of the union given by Equation \eqref{RATIONAL COMPLEMENT} consists of both the hyperplane $b_n = 0$ and the set of non-simple rational functions, $N\rat_n$.
\end{rmk*}

In the physical case, we parameterize the space of all configurations of $n$ positive masses in $\mathbb{C}$ by $n$-tuples
\begin{equation*}
\masses_n = \{ ((z_1, \sigma_1),\ldots, (z_n, \sigma_n)) \in (\mathbb{C}\times\mathbb{R}_{+})^n \mid z_i \neq z_j \textrm{ if } i \neq j\}.
\end{equation*}
Note that our parameterization represents the masses as ``marked," i.e. if two masses have the same mass and are interchanged, then the corresponding point in $\masses_n$ is different even though the physical configuration is the same. 

\begin{thm}\label{RATIONAL: PHYSICAL}
Let $S\masses_n (k)$ be the set of simple, ``positive massed" rational functions $r(z) \in \masses_n$ having degree $n \ge 2$ that yield $k$ solutions to Equation~\eqref{LENS EQUATION}. Then for all $n \ge 2$, $S\masses_n (k)$ is a non-empty, open subset of $\masses_n$ if and only if $k = n+1, n+3, \ldots, 5n-5$. Moreover, the set of non-simple positive massed rational functions, $N\masses_n$, is the complement of the union
\begin{equation*}
\bigcup_{k = n+1, n+3, \ldots, 5n-5} S\masses_n (k)
\end{equation*}
and is contained within a proper real algebraic (hence measure $0$) subset of $\masses_n$. 
\end{thm}

Thus, we have completed the solution to the problem of how many images of a star can be created by a gravitational lens consisting of $n$ point masses.

\subsection{Structure of the Paper}
In Section~\ref{SEMIALGEBRAIC}, we use basic properties of real algebraic and semialgebraic sets to show that nonsimple polynomials and rational functions lie within proper real algebraic subsets of $\poly_n$, $\rat_n$, and $\masses_n$.  The main tool used in the remainder of the paper is the extension of the ``Argument Principle" to harmonic functions $f: \mathbb{C} \to \mathbb{C}$ obtained in \cite{ARGUMENT PRINCIPLE}, which is stated precisely in Section~\ref{ARGUMENT PRINCIPLE}.  We then prove Theorem \ref{POLYNOMIAL: POSITIVE MEASURES} in Section \ref{POLYNOMIAL CASE}. In Section \ref{RATIONAL CASE}, we present Rhie's examples \cite{RHIE} and then prove Theorem \ref{RATIONAL: POSITIVE MEASURES}.  We consider the physical case in Section \ref{PHYSICAL CASE}, presenting a simplified exposition of Petters' lower bound and proving Theorem \ref{RATIONAL: PHYSICAL}.

\section{Nonsimple Harmonic Functions and Semialgebraic Geometry}\label{SEMIALGEBRAIC}

Recall that $N\poly_n$, $N\rat_n$, and $N\masses_n$ denote the sets of non-simple polynomials, rational functions, and ``positive massed" rational functions, respectively.

\begin{prop}\label{NONSIMPLE}
For any $n>1$,

\begin{itemize}
	\item $N\poly_n$ is contained in a proper real algebraic subset of $\poly_n$,
	\item $N\rat_n$ is contained in a proper real algebraic subset of $\rat_n$, and
	\item $N\masses_n$ is contained in a proper real algebraic subset of $\masses_n$.
\end{itemize}
\end{prop}

In order to prove Proposition~\ref{NONSIMPLE} we will need to take projections of real algebraic sets.  However, such projections often fail to be real algebraic sets---notable examples include the projection of $x=y^2$ or $xy=1$ on to the $x$-axis.  Thus, we will need to work in the realm of semialgebraic geometry, see~\cite{BOCHNAK} and ~\cite{COSTE}.

\begin{defn}
A \emph{semialgebraic} subset of $\mathbb{R}^n$ is a finite union of sets given by finitely many polynomial equations and inequalities with real coefficients.
\end{defn}

We will need the following key properties of semialgebraic sets:

\begin{tar*}\label{TARSKI}
Let $A$ be a semialgebraic subset of $\mathbb{R}^{n+1}$ and $\pi:\mathbb{R}^{n+1}\rightarrow\mathbb{R}^{n}$, the projection on the first $n$ coordinates.  Then $\pi(A)$ is a semialgebraic subset of $\mathbb{R}^{n}$.
\end{tar*}
\noindent See ~\cite[Thm. 2.2.1]{BOCHNAK}.

\begin{cad*}\label{CAD}
Any semialgebraic set can be decomposed into finitely many sets, each homeomorphic to $[0,1]^{d_i}$ for some $d_i$.
\end{cad*}
\noindent See ~\cite[Thm. 2.3.6]{BOCHNAK}.

\begin{defn}\label{CAD DIM}
The dimension of a semialgebraic set $A$ is the maximum of the dimensions $d_i$ from some cylindrical algebraic decomposition of $A$.  
\end{defn}

Note that the $\dim{A}$ is well-defined independent of which cylindrical algebraic decomposition is chosen.  Recall that for any $B\subset\mathbb{R}^{n}$, the Zariski closure of $B$, denoted $\zar{B}$, is the smallest real algebraic set containing $B$.

\begin{dt*}\label{SEMIALGEBRAIC DIM}For any semialgebraic set $A\subset\mathbb{R}^{n}$, $\dim{A}$ coincides with $\dim{\zar{A}}$.
\end{dt*}
\noindent
See ~\cite[Section 2.8]{BOCHNAK}.  

\begin{rmk*}The definition of dimension for a semialgebraic set $A$ is given in a different, but equivalent way in ~\cite[Section 2.8]{BOCHNAK}.  They define $\dim{A}=\dim{\zar{A}}$ and then prove that this definition coincides with the maximal dimension of any cell from the Cylindrical Algebraic Decomposition.
\end{rmk*}

\begin{proof}[Proof of Proposition~\ref{NONSIMPLE}]
Consider the real algebraic set
\begin{align*}
        V=\left\{(p,z) \in\poly_n\times\mathbb{C}\ | \ p(z)=\bar{z}, \left|p^{\prime}(z)\right|^2=1\right\}.
\end{align*}
Note that since $V$ is a real algebraic set, it is also a semialgebraic set.

If we let $\pi:\poly_n\times\mathbb{C}\rightarrow\poly_n$ be the projection onto the first coordinate.  Since $N\poly_n=\pi(
V)$, the Tarski-Seidenberg Theorem gives that $N\poly_n$ is a semialgebraic set.

We know from \cite[Lemma 5]{KHAVINSON-SWIATEK} that $S\poly_n$ is dense in $\poly_n$, hence $N\poly_n$ cannot contain a set 
homeomorphic to $[0,1]^{2n+2}$.  Thus, by Definition~\ref{CAD DIM} and the Dimension Theorem, we have that          
\begin{align*}
        \dim{\zar{N\poly_n}}=\dim{N\poly_n}\leq 2n+1.
\end{align*}
In particular, $N\poly_n$ is a subset of $\zar{N\poly_n}$, which is a proper real algebraic subset of $\poly_n$.

An identical proof shows that $N\rat_n$ and $N\masses_n$ are semialgebraic subsets of $\rat_n$ and $\masses_n$, respectively.  By the lemma from p. 1081 of ~\cite{KHAVINSON-NEUMANN}, $S\rat_n$ is a dense subset of $\rat_n$, implying that $\dim{\zar{N\rat_n}}<\dim{\rat_n}$.

Moreover, given any $r(z)\in \masses_n$ and any $c\in\mathbb{C}$, $r(z)+c$ corresponds to a shift made to the locations of each of the masses $z_i$ and therefore, $r(z)+c\in \masses_n$ as well.  Thus, the lemma from ~\cite[p. 1081]{KHAVINSON-NEUMANN} also shows that $S\masses_n$ is a dense subset of $\masses_n$, implying that $\dim{\zar{N\masses_n}}<\dim{\masses_n}$.
\end{proof}

\section{Harmonic Functions and the Argument Principle}\label{ARGUMENT PRINCIPLE}

We will need to use an extension of the classical Argument Principle to harmonic functions given by \cite{ARGUMENT PRINCIPLE} and \cite{ARGUMENT PRINCIPLE: POLES}. For a harmonic function $f(x,y) = u(x,y) + iv(x,y)$ defined on an open simply connected set $U\subset \mathbb{C}$, we can find analytic functions $h$ and $g,$ unique up to additive constants, such that $f = h + \bar{g}$. Let us consider the power-series expansions of $h$ and $g$ at $z_0$:
\begin{equation*}
h(z) = a_0 + \sum_{k=1}^{\infty} a_k(z-z_0)^k, \qquad g(z) = b_0 + \sum_{k=1}^{\infty} b_k(z-z_0)^k.
\end{equation*} 
Let $m \geq 1$ be the first index for which either $a_m$ or $b_m$ is non-zero. We say that $f$ is \emph{sense-preserving} (s.p) at $z_0$ if $a_m \neq 0$ and $|b_m/a_m| < 1$, and we say that $f$ is \emph{sense-reversing} (s.r) at $z_0$ if $b_m \neq 0$ and $|a_m/b_m| < 1$. Note that if the Jacobian $\mathcal{D}_f (z_0) = |a_1|^2 - |b_1|^2 \neq 0$, then this definition coincides with the classical one. If $h(z_0) = 0$, we define the order of $z_0$ as $+m$ if $f$ is s.p at $z_0$ and $-m$ if $f$ is s.r at $z_0$. If $f$ is neither s.p or s.r at $z_0$, then $z_0$ is called a 		\emph{singular point} and the order is undefined. 

We will also need to consider harmonic functions with poles, i.e. functions
\begin{align*}
	f: \mathbb{C} \setminus \{z_1, \ldots, z_k\} \to \mathbb{C},
\end{align*}
which are harmonic and satisfy $\lim_{z\to z_j} |f(z)| = \infty$. The points $z_1, \ldots, z_k$ are called \emph{poles} of $f$ and we will write $f(z_j) = \infty$. 

Take an oriented closed contour $\Gamma$ such that $f(x,y) \notin \{0, \infty\}$ for $(x,y) \in \Gamma$. Consider a ``normal" coordinate $s: [0,1] \to \Gamma$, with $s(0) = s(1)$, and write $f(s(t)) =  r(t)e^{i\theta(t)}$ in such a way that $\theta$ varies continuously over $[0,1]$. Then we say that $\Delta_{\Gamma} \arg(f) = \theta(1) - \theta(0) = 2\pi \cdot \omega_{\Gamma}$, where $\omega_{\Gamma}$ is an integer. We call $\omega_{\Gamma}$ the ``winding number" of $f$ over $\Gamma$. 

A simple calculation shows that the order of a zero is equal to $\omega_{\gamma}$, where $\gamma$ is a sufficiently small geometric circle centered at the zero and positively oriented. The order of a pole is defined in an ad-hoc way to be $-\omega_{\gamma}$, where $\gamma$ is defined similarly.

The result from \cite{ARGUMENT PRINCIPLE: POLES} is:

\begin{arg*}\label{ARG PRINCIPLE}
Let $f$ be harmonic, except at a finite number of poles, in a simply connected domain $D \subset \mathbb{C}$. Let $C$ be an oriented closed contour in $D$ not passing through a pole or a zero, enclosing a region (taken with orientation) $\Omega \subset D$. Suppose that $f$ has no singular zeros in $D,$ and let $N$ be the sum of the orders of the zeros of $f$ in $\Omega$. Let $M$ be the sum of the orders of the poles of $f$ in $\Omega$. Then 
\begin{equation*}
\Delta_C \arg f(z) = 2\pi N - 2\pi M.
\end{equation*}
\end{arg*}

\section{Polynomial Case}\label{POLYNOMIAL CASE}
In this section, we prove Theorem \ref{POLYNOMIAL: POSITIVE MEASURES}.  An immediate application of the Argument Principle gives the following result (briefly mentioned in \cite{KHAVINSON-SWIATEK}).
 
\begin{prop}\label{APPLICATION}
Let $p$ be a simple polynomial and let $k_+$ and $k_-$ denote the number of sense-preserving and sense-reversing zeros of $f(z) = p(z) - \bar{z}$, respectively. If $n = \deg p > 1$, then $k_+ - k_- = n$.  In particular, the total number of zeros satisfies $k=k_{+}+k_{-} \equiv n \bmod 2$.
\end{prop}
\begin{proof}
Let $S_{R}$ be a circle sufficiently large and centered at the origin so that $S_{R}$ contains all zeros of $f$ and so that $f$ applied to $S_{R}$ is dominated by the $n$-th degree term.  Then, by the Argument Principle, $k_+ - k_- = \omega_{S_{R}} = n$.
\end{proof}

\begin{proof}[Proof of Theorem \ref{POLYNOMIAL: POSITIVE MEASURES}]
We have already shown in Proposition~\ref{NONSIMPLE} that $N\poly_n$ is contained in a proper real algebraic subset of $\poly_n$ so we will focus our attention on simple polynomials. From Proposition~\ref{APPLICATION}, we have that the number of zeros $k=k_+ +k_-\geq n$ and that $k$ is congruent to $n\bmod 2$.  From \cite{KHAVINSON-SWIATEK}, we have that $k\leq 3n-2$.  Moreover by the Implicit Function Theorem, the sets $S\poly_n(k)$ are open for any $k$.  Thus, it suffices to show that $S\poly_n(k)$ is non-empty for $k=n,n+2,\ldots,3n-2$.  We will prove this by induction on $n\geq 2$. For $n = 2$, $z^2+1$ is an element of $S\poly_2(2)$ since $z^2+1=\bar{z}$ has solutions $z=-\frac{1}{2}\pm \frac{\sqrt{7}}{2}i$. Also, $z^2$ is an element of $S\poly_2(4)$ since $z^2=\bar{z}$ has solutions $z=0,1,-\frac{1}{2}\pm\frac{\sqrt{3}}{2}i$.

Suppose $S\poly_n (k) \not= \emptyset$ for $k = n, n+2, \ldots, 3n-2$. For each such $k$, we will show that $S\poly_{n+1} (k+1) \not = \emptyset$. We first let 
\begin{align*}
&c(z) = p^{\prime}(z) - na_n z^{n-1} = (n-1)a_{n-1} z^{n-2} + \ldots + a_1,\\
&d(z) = p(z) - a_n z^{n} = a_{n-1} z^{n-1} + \ldots + a_0, \textrm{ and}\\
&\widetilde{c} (z) = \frac{z}{n+1}(1+|c(z)|).
\end{align*}
Since $d(z)$ and $\widetilde{c} (z)$ both have degree $n-1$, $|d(z)| + |\widetilde{c}(z)| \le C |z|^{n-1}$ for some constant $C$.
Now, assume $p \in S\poly_n (k)$. Choose $R>0$ sufficiently large such that the following hold true:
\begin{enumerate}
\item All roots of $p(z)-\bar{z}$ are contained in the circle $S_R$ of radius $R$ centered at the origin.
\item Outside of $S_R$, $\left|\frac{a_n z^{n}}{n+1}\right| - C|z|^{n-1} > |z|$.
\end{enumerate}
Denote the $k$ roots of $f(z) = p(z) - \bar{z}$ as $r_1, \ldots, r_k$. For each $i$, let $D_{r_i}$ be a sufficiently small disk centered at $r_i$ such that the $D_{r_i}$ are pairwise disjoint and:
\begin{enumerate}[resume]
\item There exists $A$ such that for all $z\in\mathbb{C}\backslash \bigcup D_{r_i},\ |f(z)| = |p(z) - \bar{z}| > A > 0$.
\item There exists $B$ such that for all $z\in\bigcup D_{r_i},\ |\mathcal{D}(z)| = ||p'(z)|^2 - 1| > B > 0$.
\end{enumerate}
Let $\widetilde{p}(z)=p(z)+\epsilon z^{n+1}$, $\widetilde{f}(z) = \widetilde{p}(z)-\bar{z}$, and let $\widetilde{k}_+$ and $\widetilde{k}_-$ denote the number of s.p and s.r zeros of $\widetilde{f}$, respectively. We will first prove that $\widetilde{k}_- = k_-$. Using (1) and (2), we can find $\epsilon>0$ sufficiently small so that
\begin{enumerate}[resume]
\item $\widetilde{f}$ remains non-zero inside of $S_R$ and outside of $\bigcup D_{r_i}$,
\item the winding number of $\widetilde{f}$ on each of the circles $C_{r_i} = \partial D_{r_i}$ remains as $\pm 2\pi$, and
\item the Jacobians of $\widetilde{f}$ and $f$ have the same sign on each $D_{r_i}$.
\end{enumerate}
Thus, by the Argument Principle, $\widetilde{f}$ has exactly one zero within each disc $D_{r_i}$ and this zero has the same orientation as the original zero of $f$ in $D_{r_i}$. We conclude that $\widetilde{f}$ and $f$ have the same number of s.r zeros inside of $S_R$.

We now show that $\widetilde{f}$ has no s.r zeros outside of $S_R$. Outside of $S_R,$ $\widetilde{f}(z)$ is s.p, except possibly in a small region determined by the inequality $|\widetilde{p}^{\prime}(z)|^2 = |(n+1)\epsilon z^n + na_n z^{n-1} +  c(z)|^2 \le 1$. By the triangle inequality, all points in the non-s.p region satisfy
\begin{equation*}
|(n+1)\epsilon z^n + na_n z^{n-1}| \le 1 + |c(z)|.
\end{equation*}
On this region,
\begin{align*}
|\widetilde{f}(z)| &\ge |\epsilon z^{n+1} + a_n z^n| - |d(z)| - |\bar{z}| \ge \left|\frac{a_n z^{n}}{n+1}\right| - \frac{|z|}{n+1}\left|(n+1)\epsilon z^n + na_n z^{n-1}\right| - |d(z)| - |z|\\
&\ge \left|\frac{a_n z^{n}}{n+1}\right| - \frac{|z|}{n+1}(1+|c(z)|) - |d(z)| - |z| \ge \left|\frac{a_n z^{n}}{n+1}\right| - |\widetilde{c}(z)| - |d(z)| - |z|\\
&\ge \left|\frac{a_n z^{n}}{n+1}\right| - C|z|^{n-1} - |z|.
\end{align*}
By the assumption on $R$, $|\widetilde{f}(z)| > 0$ on the non-s.p region outside of $S_R$. Since $\tilde{f}$ and $f$ have the same number of s.r zeros inside of $S_R$ and neither of them has any zeros outside of $S_R$, we conclude that $\widetilde{k}_- = k_-$. 

By Proposition \ref{APPLICATION}, $\widetilde{k}_+ - \widetilde{k}_- = n + 1$, so $\widetilde{k}_+ = k_+ + 1,$ which gives that $\widetilde{k} = k+1$. Thus we have shown that $S\poly_{n+1}(k+1)\not=\emptyset$ for $k+1=n+1,n+3,\ldots,3n-1$.  It remains to show $S\poly_{n+1}(3(n+1)-2)\not= \emptyset$; however, this follows from Geyer's examples \cite{GEYER 1}.
\end{proof}

\section{Rational Case}\label{RATIONAL CASE}

In this section, we will prove Theorem~\ref{RATIONAL: POSITIVE MEASURES} using methods similar to those in the polynomial case.
Throughout the proof, we will only consider rational functions $r(z)=\frac{p(z)}{q(z)}\in\rat_n$ for which $\deg{q}\geq\deg{p}$.  In terms of the description of $\rat_n$ given in the introduction, this amounts to throwing out the algebraic hyperplane $b_n=0$; see Equation~\ref{RATIONAL FUNCTION}. 

Using the Argument Principle, we can prove the following lemma, which has previously appeared in \cite[Cor. 2]{KHAVINSON-NEUMANN}.
\begin{lem}\label{S.P-S.R}
Let $r(z) =\frac{p(z)}{q(z)}$ be a simple rational function with $\deg q \geq \deg p$ and let $k_+$ and $k_-$ denote the number of sense-preserving and sense-reversing roots of $f(z) = r(z) - \bar{z}$, respectively.  Then $k_+ - k_- = n-1.$
\end{lem}

\begin{proof}[Proof of Lemma~\ref{S.P-S.R}]
Let $f(z)=r(z)-\bar{z}$.  Consider a circle $S_{R}$ of large radius $R$ centered at the origin such that the $S_{R}$ contains all the zeros and poles of $f$. Since $\deg p \leq \deg q$, $r(z)$ is at most $O(1)$ for $z$ large. Then $f$ is sense-reversing on $S_{R}$ with an argument change of $-2\pi$. By the Argument Principle, $-2\pi = 2\pi\cdot(N-M)$, where $N = k_+ - k_-$ and $M$ is the number of poles of $f$ counted with orientation. But notice that at each of the $n = \deg q$ poles of $r$, $f$ is sense-preserving, so $-2\pi = 2\pi\cdot(k_+ - k_- -n)$, which gives $k_+ - k_- = n-1$, as desired.
\end{proof}

The examples given in Proposition \ref{5n-5 BOUND} below were previously presented by Rhie in a preprint \cite{RHIE}.  Since they were never published and since we will later use details from the construction, we will reproduce them here. 

\begin{prop}[Rhie \cite{RHIE}]\label{5n-5 BOUND}
For $n\ge 2$, $S\rat_n (5n-5)$ is a non-empty, open subset of $\rat_n$.
\end{prop}

\begin{proof}
\begin{figure}
\includegraphics[height=9cm]{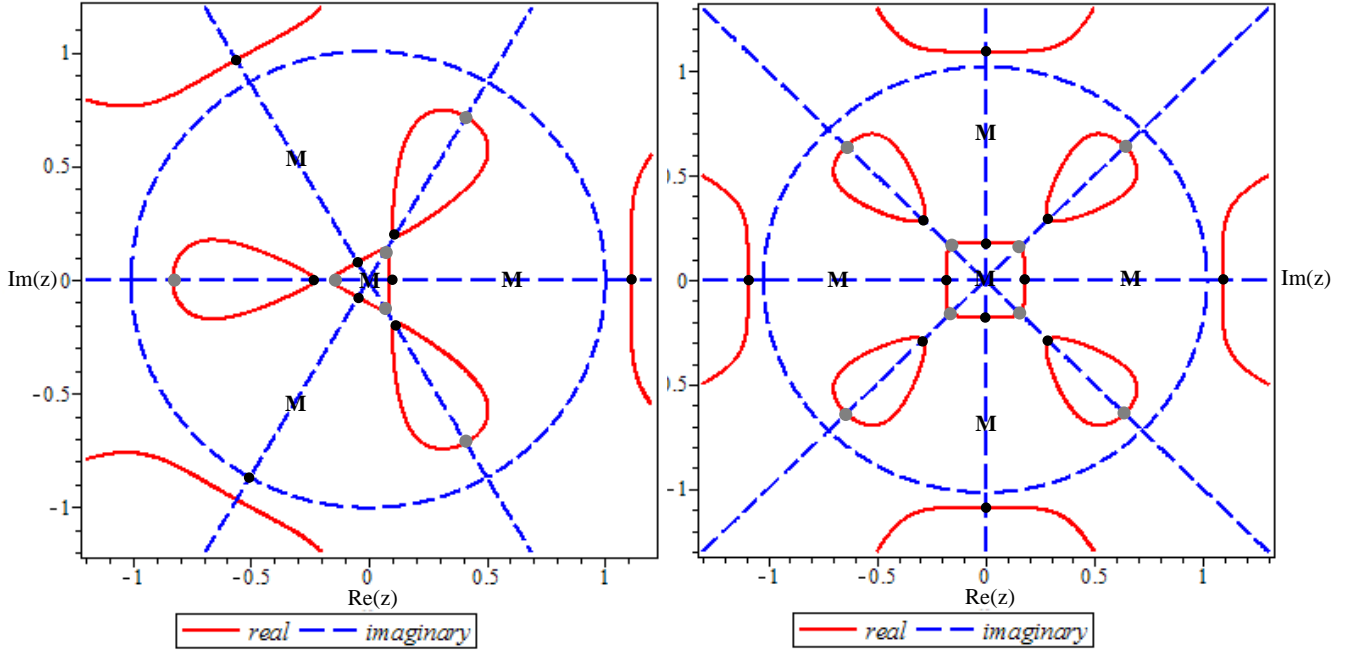}
\caption{The diagram on the left shows the real and imaginary parts of the lens equation for three point masses around a circle and an epsilon mass at the center.  The diagram on the right shows the configuration for four point masses and an epsilon mass at the center. The bold M's represent mass positions. The black dots are sense-preserving image positions and the gray dots are sense-reversing image positions.}
\label{mass configurations}
\end{figure}

We will use the Lens Equation ~\eqref{LENS EQUATION} to show that the upper bound of $5n-5$ solutions is attainable for $n\geq 2$ by a specific choice of masses $\sigma_j$ and locations of masses $z_j$. Throughout the proof, the reader may find it helpful to look at Figure 
\ref{mass configurations}, where the construction is illustrated with four and five masses.

Consider $n$ equal masses (with one mass on the positive real axis) equally spaced around a circle of radius $a = (n-1)^{-\frac{1}{n}}\left(\frac{n-1}{n}\right)^{\frac{1}{2}}$ (in fact, any small $a$ will work) centered at $0$. With $z_j=ae^{\frac{2\pi i}{n}(j-1)}$, \eqref{LENS EQUATION} becomes:
\begin{equation*}
0 = r(z) - \bar{z} ;\  r(z) = \frac{z^{n-1}}{z^n - a^n}
\end{equation*}

In the case that $n$ is odd, we only need to consider the lens equation on the real line since the mass configuration is symmetric by a rotation of $\frac{2\pi}{n}$. There is always one solution at $z=0$, and the other solutions satisfy
\begin{equation}\label{odd case}
z^n - z^{n-2} - a^n = 0.
\end{equation}
The equation $z^n - z^{n-2}$ has three real solutions, and it is easy to see that shifting this graph vertically by a sufficiently small amount ($a^n$ in our case) will still give three real solutions. Hence, we have a total of $3n + 1$ images.  (This was previously observed in \cite{MPW 97}.)

In the case that $n$ is even, we need to consider the real line and the line obtained by a rotation of $\pi/n$. Solving \eqref{odd case} gives two real solutions by an analysis similar to that in the odd case. On the rotated line, we set $z = te^{i\pi/n}$ with $t\not = 0$ and the lens equation,
\begin{equation*}
\bar{z} = \frac{z^{n-1}}{z^n - a^n},
\end{equation*}
simplifies to
\begin{equation}\label{even case}
t^n - t^{n-2} + a^n = 0.
\end{equation}
A similar analysis of this equation gives four real solutions. Therefore, there are a total of six non-zero solutions on the two neighboring lines, and the total number of images in the even case is $6 \cdot n/2 + 1 = 3n + 1$.

We now drop a small mass $\epsilon > 0$ at the center of the circle. This changes the number of masses to $n+1$, so we now require $5n$ distinct image positions. We will first check that if $\epsilon$ is sufficiently small, then all $3n$ solutions away from the origin will persist and any new solutions will occur within distance $2\sqrt{\epsilon}$ from the origin. 

Let $f(z) = r(z) - \bar{z}$. Denote the $3n+1$ roots of $f$ as $0 = r_0, \ldots, r_{3n}$. For each $i$, let $D_{r_i}$ be a sufficiently small disk centered at $r_i$ such that the $D_{r_i}$ are pairwise disjoint and
\begin{enumerate}
\item There exists $A$ such that for all $z\in \mathbb{C}\setminus \bigcup D_{r_i}, |f(z)| = |r(z) - \bar{z}| > A > 0$.
\item There exists $B$ such that for all $z\in\bigcup D_{r_i},\ |\mathcal{D}(z)| = ||r^{\prime}(z)^2| - 1| > B > 0$.
\end{enumerate}
Suppose that $\epsilon$ is sufficiently small so that $A > \sqrt{\epsilon}$ and so that the disk $D^{\prime}_{r_0}$ of radius $2\sqrt{\epsilon}$ centered at $0$ satisfies $\overline{D^{\prime}_{r_0}} \subset D_{r_0}$. By (2), there are no critical points in the annulus $\overline{D_{r_0}}\setminus D^{\prime}_{r_0}$, so the minimum of $|f|$ (over this annulus) occurs on its boundary. On $\partial D_{r_0}$, $|f(z)| > B$, and on $\partial D^{\prime}_{r_0},$
\begin{equation*}
|f(z)| = |r(z) - \bar{z}| \ge |\bar{z}| - |r(z)| \ge \sqrt{\epsilon}
\end{equation*}
since $r(0) = r'(0) = 0$. Hence, we have:

\begin{itemize}
\item[($1^\prime$)] For all z $\in \mathbb{C} \setminus D^{\prime}_{r_0} \cup D_{r_1} \cup \cdots \cup D_{r_{3n}}$,\, $|f(z)|> \sqrt{\epsilon}$.
\end{itemize}

Adding the mass $\epsilon$ results in $\widetilde{f}(z) = \widetilde{r}(z) - \bar{z} = r(z) + \frac{\epsilon}{z} - \bar{z}$. By ($1^\prime$) and (2), if $\epsilon$ is sufficiently small, then in the region $|z| \ge 2\sqrt{\epsilon}$, each of the zeros of $\widetilde{f}$ is obtained by a continuous motion of a zero of $f$, and each of them will remain simple with the same orientation. 

We now focus on $|z| < 2\sqrt{\epsilon}$. The new lens equation on the real line is
\begin{equation}\label{NEW IMAGES}
z^{n+2} - z^n(1+\epsilon) - z^2 a^n + \epsilon a^n = 0.
\end{equation}
At this scale, equation \eqref{NEW IMAGES} is an arbitrarily small perturbation of the equation $-z^2 a^n + \epsilon a^n = 0$, which gives two simple solutions at $\pm \sqrt{\epsilon}$. If $n$ is odd, then symmetry under rotation by $2\pi/n$ leads to a total of $2n$ new simple images in the $2\sqrt{\epsilon}$ neighborhood of the origin.  Since we originally had $3n+1$ simple zeros, we now
have exactly exactly $3n + 1 + 2n - 1 = 5n$ simple zeros, as desired. 

If $n$ is even, symmetry under rotation by $2\pi/n$ only leads to $n$ new images in the $2\sqrt{\epsilon}$ neighborhood of the origin. However, on the line $z=t e^{i \pi/n}$ the lens equation is the same as Equation \eqref{NEW IMAGES}, except that the first two terms have the opposite sign.  Two new simple zeros are produced on this line at approximately $t = \pm \sqrt{\epsilon}$.   Under symmetry by rotation of $2\pi/n$, this produces the $n$ additional new images that we needed.

In the case of $n = 2$ and $n = 3$, this construction fails due to the low number of masses. However, a quick check shows that the original construction of $n$ equally spaced masses around the circle of radius $a$ gives $5$ and $10$ simple zeros in these two cases. For example, in the case of $n = 2$, we have $r(z) = \frac{z}{z^2 - 1/2}$, and solving the equation $\bar{z} = r(z)$ will give five solutions. The reason we do not reach the $3n + 1$ bound in the case $n = 2$ is because on the imaginary axis, there are only two solutions (since the degree is not high enough). In the case of $n = 3$, $3n + 1 = 5n-5$, so we see that the original construction works as well.
\end{proof}

\begin{prop}\label{RATIONAL: 5n-7}
For $n\geq 2$, $S\rat_{n}(5n-7)$ is a non-empty, open subset of $\rat_n$.
\end{prop}

\begin{proof}[Proof of Proposition~\ref{RATIONAL: 5n-7}]
Placing mass $1/2$ at $z = 1$ and $z = -1$ produces three images. Placing mass $1$ at $z = \frac{1}{2}$ and $z = -\frac{1}{2} \pm \frac{\sqrt{3}}{2}i$ produces eight images. 

For four or more masses, the construction is a perturbation of Rhie's examples, which consisted of $n$ equal masses equally spaced around the circle of radius $a$ and an $\epsilon$ mass at the origin. We claim that a small perturbation for the location of the $\epsilon$ mass from the origin will result in $5n-2$ simple solutions. More exactly, for appropriate choices of $A$ and $\phi$, we claim that the equation 
\begin{equation}\label{NEW LENS}
\frac{z^{n-1}}{z^n - c} + \frac{\epsilon}{z-b} = \bar{z},
\end{equation} 
where $b = A\epsilon^{(n-1)/2}\cdot e^{i \phi}$ and $c = a^n$ is the constant from Proposition~\ref{5n-5 BOUND}, has exactly $5n-2$ simple solutions so long as $\epsilon > 0$ is sufficiently small.

As in the Proof of Proposition \ref{5n-5 BOUND}, a perturbation by $\frac{\epsilon}{z-b}$ of $f(z) = r(z) - \bar{z}$ will not affect the roots outside of the disk of radius $2\sqrt{\epsilon}$ centered at the origin so long as $|b| < \sqrt{\epsilon}$. Hence, we need only consider Equation~\eqref{NEW LENS} inside the disk of radius $2\sqrt{\epsilon}$. 
 
Equation~\eqref{NEW LENS}, after simplification, becomes 
\begin{equation}\label{NEW LENS 1}
\frac{z^n}{z^n - c} + \epsilon = z\bar{z} - b\bar{z} + \frac{bz^{n-1}}{z^n - c}.
\end{equation}
Replacing $z^n - c$ by $-c$ in Equation~\eqref{NEW LENS 1} gives
\begin{equation}\label{APPROXIMATION}
\epsilon - \frac{z^n}{c} = z\bar{z} - b\bar{z} - \frac{bz^{n-1}}{c},
\end{equation}
which is an approximation of Equation~\eqref{NEW LENS 1} up to an error term of order $\epsilon^{n}$. Let $z = re^{i\theta}$. The real part of Equation~\eqref{APPROXIMATION} gives $r = \sqrt{\epsilon} + \mathcal{O}(\epsilon)$. We now compare the imaginary parts of both sides, which will give us the values of $\theta$. 

If $n$ is even, let $\phi = \pi/2 + \pi/2n$. The lowest order terms for the imaginary part of Equation~\eqref{APPROXIMATION} are of order $\epsilon^{n/2}$. They are
\begin{equation*}
\Imag\left(-\frac{z^n}{c}\right) = \Imag\left(-b\bar{z}\right).
\end{equation*}
Using that $r = \sqrt{\epsilon} + \mathcal{O}(\epsilon)$ and $|b| = A\epsilon^{(n-1)/2}$, the lowest order terms give
\begin{equation*}
\sin n\theta = C\sin(\phi - \theta),
\end{equation*}
where $C = A\cdot c$. We notice that when $\theta = \pi/2n$, both sides attain their maximum value. The left hand side attains its minimum value when $\theta = \frac{(4k+3)\pi}{2n}$ for integer $k$. The right hand side attains its minimum when $\theta = \frac{(-4kn - 2n + 1)\pi}{2n}$ for integer $k$. However, $4k + 3 \not\equiv -4kn - 2n + 1 \pmod 4$ since $n$ is even. Therefore, $\sin n\theta$ and $\sin(\phi-\theta)$ are tangent at only one point, $\theta = \pi/2n$. When $C=0$, there are $2n$ solutions - a simple calculation shows that for $C < 1$, $\sin n\theta$ and $\sin(\phi-\theta)$ cannot be tangent. Hence, the first time a tangency occurs is when $C = 1$, and from above, this tangency is located at the single point $\theta = \pi/2n$. Increasing $C$ by a sufficiently small amount will pull apart this tangency, but the other simple solutions to $\sin n\theta = \sin(\phi-\theta)$ will persist, so we have lost exactly two solutions, as desired.

When $n$ is odd, for any choice of $\phi$, the above approach produces two tangencies and thus it cannot be used without considering higher order terms. Let $\phi = \pi/2$. As before, $r = \sqrt{\epsilon} + \mathcal{O}(\epsilon)$. The imaginary part of Equation~\eqref{APPROXIMATION} yields
\begin{equation}\label{IMAGINARY}
\sin n\theta = \frac{c|b|}{r^{n-1}} \cos \theta + \frac{|b|}{r} \cos((n-1)\theta).
\end{equation}
Up to terms of order $\epsilon^{(n-1)/2}$, Equation~\eqref{IMAGINARY} becomes
\begin{equation*}
\sin n\theta = C \cos \theta + A\epsilon^{(n-2)/2}\cos ((n-1)\theta),
\end{equation*}
where $C = A\cdot c$. We first look at the lowest order terms, which gives
\begin{equation}\label{SIMPLE EQUATION}
\sin n\theta = C\cos\theta.
\end{equation}
When the mass $\epsilon$ is located exactly at the center, $C = 0$, and there are $2n$ solutions. First, notice that if $\theta$ is a solution to \eqref{SIMPLE EQUATION}, so is $\pi + \theta$. Increasing $C$ will not change the number of solutions to \eqref{SIMPLE EQUATION} until $C \cos \theta$ is tangent to $\sin n\theta$ at some point. For this to occur, we require a solution to \eqref{SIMPLE EQUATION} to also be a solution to
\begin{equation}\label{TANGENCY}
n\cos n\theta = -C \sin \theta.
\end{equation}
Squaring $n\cdot$\eqref{SIMPLE EQUATION} and \eqref{TANGENCY} and summing gives 
\begin{equation*}
n^2 = C^2 + C^2(n^2-1)\cdot \cos^2 \theta,
\end{equation*}
which gives two solutions for $\cos \theta$ and four for $\theta$ - however, one can readily verify that two of these are extraneous, so there are only two values of $\theta$ for which \eqref{SIMPLE EQUATION} and \eqref{TANGENCY} are satisfied, and moreover they differ by $\pi$. Reconsidering the small term $A\epsilon^{(n-2)/2} \cos((n-1)\theta)$, we find that $\cos ((n-1)\theta) = \cos ((n-1)(\pi + \theta))$. Since one of the tangencies was a minimum for both sides of Equation~\eqref{SIMPLE EQUATION} and the other was a maximum, this perturbation will pull apart one of the tangencies and change the other to be two simple solutions. 
\end{proof}

With these facts, we can prove Theorem \ref{RATIONAL: POSITIVE MEASURES}.

\begin{proof}[Proof of Theorem \ref{RATIONAL: POSITIVE MEASURES}]
We have already shown in Proposition~\ref{NONSIMPLE} that $N\rat_n$ is contained in a proper real algebraic subset of $\rat_n$ so we will focus our attention on simple rational functions. 

We will also throw out the proper algebraic set given by $b_n=0$ in the parameterization of rational functions (Equation ~\eqref{RATIONAL FUNCTION}). This allows us to restrict our attention to the case that $\deg q \geq \deg p$.  Under these assumptions, we have that the number of zeros, $k$, satisfies $n-1\leq k$, from the Argument Principle, and $k\leq 5n-5$, from ~\cite{KHAVINSON-SWIATEK}. Furthermore, from Lemma ~\ref{S.P-S.R}, we have that $k\equiv n-1\bmod 2$.  Thus, it suffices to show that for each $n\geq 2$, $S\rat_n(k)$ is non-empty for $k=n-1,n+1,n+3,\ldots,5n-5$.

We will proceed with an induction on $n$, similar to the proof of Theorem~\ref{POLYNOMIAL: POSITIVE MEASURES}.
First consider the case where $n = 2$. The equation $r(z) = \frac{c}{z^2}=\bar{z}$ has exactly one root for any $c$ so $S\rat_2(1)$ is non-empty.  Furthermore, by Proposition~\ref{5n-5 BOUND} and Proposition~\ref{RATIONAL: 5n-7} we have that $S\rat_2(3)$ and $S\rat_2(5)$ are non-empty, respectively. 

Now, suppose for some $n\geq 2$ that $S\rat_n (k)\neq\emptyset$ for $k=n-1,n+1,n+3,\ldots,5n-5$. We will first show that $S\rat_{n+1} (k+1)\neq\emptyset$ for the same values of $k$. By assumption, there exists $r(z) \in S\rat_n(k)$. Denote the $k$ roots of $f(z) = r(z) - \bar{z}$ as $r_1, \ldots, r_k$ and the $n$ poles of $f$ as $z_{1}, z_{2}, \ldots, z_{n}$. Consider the rational function $\widetilde{r} = r(z) + \frac{\epsilon}{z-\widetilde{z}}$, where $\widetilde{z} \not = r_i$ and $\widetilde{z} \not= z_i$ for any $i$. For each $i$, let $D_{r_i}$ be a sufficiently small disk centered at $r_i$ and let $D_{z_i}$ be a sufficiently small disk centered at $z_i$ such that the $D_{r_i}$ and $D_{z_i}$ are pairwise disjoint, $\widetilde{z}\notin \bigcup D_{r_i} \bigcup D_{z_i}$, and
\begin{enumerate}
\item There exists $A$ such that for all $z\in\mathbb{C}\backslash \bigcup D_{r_i},\ |f(z)| = |\bar{z}-r(z)| > A > 0$.
\item There exists $B$ such that for all $z\in\bigcup D_{r_i},\ |\mathcal{D}(z)| = ||r^{\prime}(z)^2| - 1| > B > 0$.
\item There exists $C$ such that for all $z\in\mathbb{C}\backslash \bigcup D_{z_i},\ |r'(z)| < C$.
\end{enumerate} 
Clearly, $\widetilde{r}$ has degree $n+1$, so it is sufficient to show that $\widetilde{f} = \widetilde{r}(z) - \bar{z}$ has $k+1$ zeros. For $\epsilon>0$ sufficiently small, the winding number around each of the circles $C_{r_i} = \partial D_{r_i}$ remains as $\pm 2\pi$ and the Jacobian does not change sign in $D_{r_i}.$ Therefore, each root $r_i$ of $f$ moves continuously to some new simple root $\widetilde{r}_i$ of the same ``orientation" as $r_i$. Moreover, $\widetilde{r}_i$ remains in the disk $D_{r_i}$, so there is still exactly one root in each disk.

We claim now that any new root of $\widetilde{f}$ must be sense-preserving. If $\widetilde{f}(z) = 0$, then \begin{equation*}
\left|\frac{\epsilon}{z-\widetilde{z}}\right| = |\bar{z} - r(z)| > A \Rightarrow \left|\frac{\epsilon}{(z-\widetilde{z})^2}\right| > \frac{A^2}{\epsilon}.
\end{equation*} 
At any such point, 
\begin{equation*}
|\widetilde{r}^{\prime}(z)| = \left|r^{\prime}(z) - \frac{\epsilon}{(z-\widetilde{z})^2}\right| > \left|\frac{\epsilon}{(z-\widetilde{z})^2}\right| - |r^{\prime}(z)|  > \frac{A^2}{\epsilon} - C.
\end{equation*}
As long as $\epsilon$ is sufficiently small, $|\widetilde{r}^{\prime}(z)| > 1$, so any new root of $\widetilde{f}$  is sense-preserving. By Lemma \ref{S.P-S.R}, $\widetilde{k}_+ - \widetilde{k}_- = n$. As no new s.r roots are created, $k_-$ stays the same, so $k_+$ must have increased by exactly $1$, as desired.

We have shown that $S_{n+1}(k)$ is non-empty for $k=n,n+2,\ldots,5n-4 =5(n+1) - 9$.  However, Propositions \ref{5n-5 BOUND} and \ref{RATIONAL: 5n-7} give that $S_{n+1}(5n-2)$ and $S_{n+1}(5n)$ are also non-empty, thus our proof is complete.
\end{proof}

\section{Physical Case}\label{PHYSICAL CASE}

The Lens Equation ~\eqref{LENS EQUATION} can be rewritten as
\begin{equation}\label{LENS EQ}	
0=\bar{z}-\sum_{i=1}^{n}{\frac{\sigma_{i}(\bar{z}-\bar{z_{i}})}{|\bar{z}-\bar{z_{i}}|^2}},
\end{equation}
where $0$ is the source position, $z_{i}$ is the position of the $i$th point mass, and $z$ is the image position.

The \emph{time delay function} $T$ is defined to be
\begin{align*}
	T(z)=\frac{\left|z\right|^2}{2}-\sum_{i=1}^{n}{\sigma_{i}\ln{|z-z_{i}|}}.
\end{align*}

Note that $z_0=x+iy$ is a solution to $\partial_{z}T=(T_{x}-iT_{y})/2=0$ (a critical point of $T$) if and only if $z$ is a solution to Equation~\ref{LENS EQ}. Moreover, $z_0$ is a non-degenerate local minimum (respectively maximum) of $T$ iff $z_0$ is a s.r (respectively s.p) simple zero of \eqref{LENS EQ}.

\begin{prop}[Petters \cite{PETTERS2}]\label{LOWER BOUND}
If all of the masses $\sigma_i$ are positive and if all of the zeros of the lens equation \eqref{LENS EQ} are simple, then there are at least $n+1$ of them.
\end{prop}

\begin{proof}
Let $\overline{D_R}$ be the closed disc centered at the origin of radius $R$ and for each $i$, let $D_{i,\epsilon}$ be the open disc of
radius $\epsilon$ centered at $z_i$. Let $G = \overline{D_R} \setminus \cup D_{i,\epsilon}$. 

If $\left|z\right|\to\infty$ or if $z\to z_i$, then $T(z) \to \infty$. Hence, we can choose $R > 0$ sufficiently large and $\epsilon > 0$ sufficiently small so that the minimum of $T$ is not attained on the boundary of $G$. Therefore, the minimum of $T$ occurs on the interior of $G$. This minimum corresponds to a s.r simple zero of \eqref{LENS EQ}. We can show that $n-1=k_{+}-k_{-}$ as in Lemma \ref{S.P-S.R}, where $k_{+}$ ($k_-$) is the number of s.p (s.r) zeros of \eqref{LENS EQ}. Since we have at least one s.r zero, $k_- \ge 1$, and $k_{+}+k_{-}\ge n-1 + 2k_- \ge n+1$, as desired.
\end{proof} 

\begin{proof}[Proof of Theorem~\ref{RATIONAL: PHYSICAL}] 
We use an inductive proof similar to that of Theorem~\ref{RATIONAL: POSITIVE MEASURES}.

Consider the case $n=2$. Placing mass $1$ at any point other than the origin will result in two images and the inductive step from the proof of Theorem~\ref{RATIONAL: POSITIVE MEASURES} shows that placing a sufficiently small $\epsilon$ mass anywhere else will generate exactly one more image.  Also, Rhie~\cite{RHIE} gives an example for a configuration that yields $5$ roots.  Thus, $S\masses_{2}(3)$ and $S\masses_{2}(5)$ are non-empty open sets.

From here, exactly the same inductive step can be applied as in the proof of Theorem~\ref{RATIONAL: POSITIVE MEASURES} to show that if $S\masses_n (n+1), \ldots, S\masses_n(5n-5)$ are non-empty, then $S\masses_{n+1}(n+3), \ldots, S\masses_{n+1}(5n-4)$ are non-empty since addition of the term $\epsilon/(z-\widetilde{z})$ corresponds to adding a small mass at $\widetilde{z}$. Propositions~\ref{5n-5 BOUND} and \ref{RATIONAL: 5n-7} give that $S\masses_{n+1}(5n-2)$ and $S\masses_{n+1}(5n)$  are non-empty.
\end{proof}

\begin{rmk*}
Each of the $S\masses_{n}(k)$ contains a positive massed rational function with masses arbitrarily close to the origin. Thus, these functions are consistent with the physical assumption that all the masses are close together and at small angular positions with respect to the light source. In the proof of Theorem~\ref{RATIONAL: PHYSICAL}, the base case of the induction can be done with two masses arbitrarily close to the origin. The inductive step from the proof of Theorem~\ref{RATIONAL: POSITIVE MEASURES} allowed us to add a small mass at an arbitrary point and, moreover, the proofs of Propositions~\ref{5n-5 BOUND} and \ref{RATIONAL: 5n-7} can be done with masses arbitrarily close to the origin.
\end{rmk*}

\vspace{0.1in}

\noindent\textbf{Acknowledgements.}
We thank Dmitry Khavinson for useful discussions about the problem and the referee for his or her helpful remarks.
The work of Bleher is supported in part by the NSF grant DMS-0969254. The work of Roeder is supported in part by the NSF grant DMS-1102597 and startup funds from the Department of Mathematics at IUPUI.


 



\begin{thebibliography}{00}


\bibitem{BOCHNAK}
Jacek Bochnak, Michel Coste, Marie-Fran\c{c}oise Roy.
 \textit{Real Algebraic Geometry.}
 Springer, Berlin (1998).

\bibitem{BURKE}
W. Burke.
	Multiple Gravitational Imaging by Distributed Masses. 
	\textit{The Astrophysical Journal.}
	Volume 244, Page L1, 1981.

\bibitem{COSTE}
Michel Coste.
 \textit{An Introduction to Semialgebraic Geometry.}
 Institut de Recherche Math\'{e}matique de Rennes (2002).

\bibitem{ARGUMENT PRINCIPLE}
P. Duren, W. Hengartner, R. S. Laugesen.
 The Argument Principle for Harmonic Functions.
 \textit{The American Mathematical Monthly.}
 Volume 103, Number 5, Pages 411-415, May 1996. 

\bibitem{FASSNACHT}
C. D. Fassnacht, C. R. Keeton, D. Khavinson.
	Gravitational Lensing by Elliptical Galaxies, and the Schwarz Function.
	\textit{Analysis and Mathematical Physics.}
	Pages 115-129, Trends in Mathematics, Birkhäuser, Basel, 2009.

\bibitem{GEYER 1}
L. Geyer.
 Sharp Bounds for the Valence of Certain Harmonic Polynomials.
 \textit{Proceedings of the American Mathematical Society.}
 Volume 136, Number 2, Pages 549-555, February 2008.

\bibitem{KHAVINSON-NEUMANN}
D. Khavinson, G. Neumann.
 On the Number of Zeros of Certain Rational Harmonic Functions.
 \textit{Proceedings of the American Mathematical Society.}
 Volume 134, Number 4, Pages 1077-1085, July 2005.

\bibitem{KHAVINSON-NEUMANN2}
D. Khavinson, G. Neumann.
 From the Fundamental Theorem of Algebra to Astrophysics: A ``Harmonious" Path.
 \textit{Notices of the American Mathematical Society.}
 Volume 55, Number 6, Pages 666-675, June/July 2008.

\bibitem{KHAVINSON-SWIATEK}
D. Khavinson, G. \'{S}wi\c{a}tek.
 On the Number of Zeros of Certain Harmonic Polynomials.
 \textit{Proceedings of the American Mathematical Society.}
 Volume 131, Number 2, Pages 409-414, September 2002.
 
\bibitem{MPW 97}
S. Mao, A. O. Petters, H. J. Witt.
	Properties of Point Mass Lenses on a Regular Polygon and the Problem of Maximum Number of Images.
	\textit{The Eighth Marcel Grossmann Meeting, Part A, B (Jerusalem, 1997).}
	Pages 1494-1496, World Sci. Publ., River Edge, NJ, 1999.

\bibitem{PETTERS2}
A. O. Petters.
 Morse Theory and Gravitational Microlensing.
 \textit{Journal of Mathematical Physics.}
 Volume 33, Number 5, Pages 1915-1931, May 1992.

\bibitem{PETTERS_SURVEY}
A. O. Petters.
 Gravity's Action on Light.
 \textit{Notices of the AMS.}
 Volume 57, Number 11, Pages 1392-1409, December 2010.

\bibitem{PETTERS}
A. O. Petters, H. Levine, J. Wambsganss.
 \textit{Singularity Theory and Gravitational Lensing.}
 Birkh\"{a}user, Boston (2001).

\bibitem{PETTERS-WERNER}
A. O. Petters, M. C. Werner.
	Mathematics of Gravitational Lensing: Multiple Imaging and Magnification.
	\textit{General Relativity Gravitation.}
	Volume 42, Number 9, Pages 2011-2046, 2010.

\bibitem{RHIE}
S. H. Rhie.
 n-point Gravitational Lenses with 5(n-1) Images.
 arXiv:astro-ph/0305166, 10 May 2003.

\bibitem{SCHNEIDER}
P. Schneider, J. Ehlers, E. E. Falco.
	\textit{Gravitational Lenses.}
	Springer, Berlin (1999).

\bibitem{STRAUMANN}
N. Straumann.
	Complex Formulation of Lensing Theory and Applications.
	 \textit{Helvetica Physica Acta.}
	 Volume 70, Number 6, Pages 894-908, 1997.

\bibitem{ARGUMENT PRINCIPLE: POLES}
T. J. Suffridge, J.W. Thompson.
 Local Behavior of Harmonic Mappings.
 \textit{Complex Variables, Theory and Application: An International Journal.}
 Volume 41, Pages 63-80, 2000. 

\bibitem{WILMHURST}
A. S. Wilmhurst.
 The Valence of Harmonic Polynomials.
 \textit{Proceedings of the American Mathematical Society.}
 Volume 126, Number 7, Pages 2077-2081, July 1998.
\end{thebibliography}
\end{document}